\documentclass[final,3p]{elsarticle}

\usepackage{lineno, hyperref}
\modulolinenumbers[5]

\usepackage{amssymb}
\usepackage{amsmath}
\usepackage{mathtools}
\usepackage{amsthm}
\usepackage{color}

\newtagform{roman}[]()

\newtheorem{thm}{Theorem}

\newdefinition{rmk}{Remark}
\newdefinition{ex}{Example}
\newdefinition{de}{Definition}
\newproof{pf}{Proof}

\newcommand{\g}{\mathfrak{g}}

\newcommand{\otrzy}{\mathfrak{o}(3)}
\newcommand{\hatrzy}{\mathfrak{h}(3)}
\newcommand{\eseldwa}{\mathfrak{sl}(2,\mathbb{R})}

\begin{document}

\begin{frontmatter}

\title{On normality of f.pk-structures on $\mathfrak{g}$-manifolds}

\author[mymainaddress]{Andrzej Czarnecki}
\cortext[mycorrespondingauthor]{Corresponding author}
\ead{andrzejczarnecki01@gmail.com}

\author[mymainaddress]{Marcin Sroka}
 
\ead{marcin.sroka@student.uj.edu.pl}

\author[mymainaddress]{Robert Wolak}

\ead{robert.wolak@im.uj.edu.pl}

\address[mymainaddress]{Jagiellonian University, \L ojasiewicza 6, 30-348 Krakow, Poland}

\begin{abstract}
We consider higher dimensional generalisations of normal almost contact structures, the so called f.pk-structures where parallelism spans a Lie algebra $\mathfrak{g}$ (f.pk-$\mathfrak{g}$-structures).Two types of these structures are discussed. In the first case, we construct an almost complex structure on a product manifold mirroring $\mathcal{K}$-structures. We show that the natural normality condition can be satisfied only when $\mathfrak{g}$ is abelian. The second case we consider is when the Lie algebra in question is 3-dimensional, but the almost complex structure on a product is constructed in a different manner. In both cases the normality conditions are expressed in terms of the structure tensors.
\end{abstract}


\end{frontmatter}

\section{Introduction}

K\"ahler manifolds and complex or almost complex manifolds form a very important and well established field of research. However, their even dimension proves an hindrance as numerous physical applications call for study of their odd-dimensional counterparts. Just like the contact topology complements symplectic topology, in the richer environment of differential geometry we study the Sasakian structures or K-contact structures on odd-dimensional manifold. In that case, we have a K\"ahler (or complex, or almost complex) structure in a direction transverse to the foliation given by the Reeb vector field, cf. \cite{sa,bl,bg}. Note however that these geometries are more rigid, even a non-vanishing Killing vector field on a compact Riemannian manifold defining a transversely K\"ahler foliation does not necessarily define a Sasakian structure, cf. \cite{bg, cap, ha, mo}. Many obstructions to the existence of such structures have been found, cf. \cite{bg} and references therein. The classical definition in \cite{bl} imposes a normality condition on the contact metric structure that may at first seem somewhat artificial.

The notion of normal almost contact structures is well known (detailed information can be found in \cite{bl}). We recall the main idea. Suppose we have an \emph{almost contact structure}: a smooth manifold $M$ together with structure tensors -- a 1-form $\eta$, a vector field $\xi$ and a tensor $\phi$ of type $(1,1)$, such that

$$
\phi^2=-I+\eta \otimes \xi, \qquad \eta(\xi)=1.
$$ 

We say that an almost contact structure $(M,\phi,\eta,\xi)$ is \emph{normal} if the almost complex structure defined on $M \times \mathbb{R}$ by

\begin{align*}
J(\xi ) &= \partial_t\\
J(\partial_t) &= -\xi\\
J(X) &= \phi (X) \qquad X \in D = im{}\phi
\end{align*}

or, equivalently, $J(X,f \frac{\partial}{\partial t})=(\phi(X)-f\xi,\eta(X) \frac{\partial}{\partial t})$, is integrable. It turns out, cf. \cite{sa,bl}, that being a normal almost contact manifold is equivalent to a simple condition, namely

$$
[\phi,\phi](X,Y)+d\eta(X,Y) \xi=0
$$

where $[\phi,\phi]$ is the Nijenhuis tensor (of type $(1,2)$) given by 

$$
[\phi,\phi](X,Y)=\phi^2[X,Y]+[\phi X,\phi Y]-\phi[\phi X,Y]-\phi[X, \phi Y].
$$

In the presence of an additional metric structure we have, for example, that a contact metric structure $(M, g, \phi , \xi , \eta )$ is Sasakian if and only if the corresponding structure on $M \times \mathbb{R}$ is conformally K\"ahler cf. \cite{bl}. We can find plenty more similar structures in differential geometry, e.g. $\mathcal{K}$-structures in \cite{dkw, dug} or 3-Sasakian structures in \cite{bg-3sas}. In these cases we have a foliation spanned by a finite dimensional Lie algebra of linearly independent vector fields, transverse structure of which is  K\"ahler or hyper-K\"ahler and we impose some additional condition of the \emph{normality type}. Geometry of such manifolds, called $\g$-\emph{manifolds}, was first studied by D. Alekseevsky et al. in \cite{A}. 

In this short paper we investigate the normality condition \'a la Sasaki for f.pk-$\g$-structures. Examples of such structures come from a locally free action of a Lie group preserving some transverse almost complex structure. We first show that the immediate generalisation of the almost complex structure appropriate for almost contact or $\mathcal{K}$-structure imposes severe restrictions on the Lie algebra $\g$. Then we turn to the 3-dimensional case and give a better suited almost complex structure and a satisfying condition for its integrability.

\section{Normal f.pk-$\mathfrak{g}$-structures}

Let $M$ be a smooth manifold of dimension $2n+p$. In analogy to almost contact structures presented in the introduction, we consider a \emph{f-structure} $\phi$ on $M$, that is a tensor of type $(1,1)$ such that

$$
\phi^3 + \phi = 0
$$

We will denote $im{}\phi$ by $D$. If $ker{}\phi$ is a $p$ dimensional parallelisable distribution and a parallelism $\xi_i^*$ is chosen, we say that $(M,\phi,\xi_i^*)$ is an \emph{f.pk-structure} i.e. a $Gl(n,\mathbb{C})\times I_p$ structure on $TM$. Throughout the paper we will assume that the parallelism spans a Lie algebra $\g$, which means that there exists a Lie algebras monomorphism 

$$
^*\,:\,\g \longrightarrow \chi(M)
$$

from a $p$-dimensional Lie algebra $\g$ into vector fields on $M$ such that image of $^*$ is $ker{}\phi$. Throughout the paper $\xi_i$ denotes a basis of $\g$ such that $^*(\xi_i)=\xi_i^*$ and we also write $\xi^*$ for $^*(\xi)$ of $\xi \in \g$ and $\g^*$ for $^*(\g)$. This somewhat cumbersome notation is meant to accommodate for that our main motivation and examples come from actions of Lie group on a manifold (and it is then customary to mark the transported invariant fields with $\ast$), however it is $\xi_i^*$ that are fixed and inherent to the structures we consider, not $^*$. We call $(M, \phi, \xi_i^*)$ an \emph{f.pk}-$\g$-\emph{structure}. We then have a decomposition $TM = ker{}\phi \oplus im{}\phi = span\{\xi_1^*,...,\xi_p^* \} \oplus D $ and we define one forms $\eta_i$ by $\eta_i\xi_j^*= \delta^i_j$ and $\left.\eta_i\right|_D=0$. We also define a 1-form on $M$ with values in $\g$ by

$$
\eta X= \sum_{i=1}^p\eta_i (X) \xi_i.
$$  

and $\eta^*$ by $\eta^*X=(\eta X)^*$. The condition $\phi^3+ \phi = 0$ can be now rewritten

$$
\phi^2=-I+\eta^*
$$

Suppose $G$ is a simply-connected Lie group with its Lie algebra $\g$. Of course (by a slight abuse of notation) $T(M \times G)=D \oplus  span\{ \xi_i^* \} \oplus  span\{ \xi_i \}$. We will write $\xi^*$, $X$ or $\xi$ for the vector fields in the respective sub-bundles, but we will use notation $(X,\xi)$ when -- and only when -- need arises to discern $TM$- and $TG$-part of a vector field on $M\times G$.

Let us consider an almost complex structure on $M \times G$ given by the formula 

$$
J(X,\xi)=\left(\phi X-\xi^*,\eta X \right)
$$ 

where $X \in \chi(M)$ and $\xi \in \chi(G)$. It is clear that it is an almost complex structure -- $J$ equals $\phi$ on $D$ and $\left.\phi\right|_{D}^2 =-I$ and on the remaining part of $T(M \times G)$ $J$ acts by

$$
\xi_i^* \mapsto \xi_i \mapsto -\xi_i^*.
$$

We say that $(M,\phi,\xi_i^*)$ is a \emph{normal} f.pk-$\g$-structure (by an abuse of notation we will sometimes say that $\phi$ is normal, yet we understand that a parallelism is fixed) if an almost complex structure $J$ is integrable. Observe that although the definition of $J$ depends on $^*$, the integrability does not. In a slightly different context it was proved in \cite{mil} that for $G= \mathbb{R}^p$, $J$ is integrable if and only if $(M,\phi,\xi_i^*)$ is an integrable f-contact manifold, yet no characterisation in terms of structure tensors was given. An abelian case was also studied in \cite{blf} (see also a rich reference list there) or as a special case in \cite{wais}.

By the celebrated theorem of Newlander-Nirenberg, it is enough to check whether

$$
[J,J] \equiv  0.
$$

Since the Nijenhuis tensor is an antisymmetric tensor it is enough to check that condition only on pairs of vector fields of three types: $(\xi_i^*,0), (X,0), (0,\xi_i)$. This gives 6 cases to consider.

\usetagform{roman}

\begin{align}
\begin{aligned}
\left[J,J \right] \left[(\xi_i^*,0),(\xi_j^*,0) \right]=&-\left([\xi_i^*,\xi_j^*],0 \right)+ \left([\phi \xi_i^*,\phi \xi_j^*],[\xi_i,\xi_j] \right)\\
&-J \left[(\phi \xi_i^*,\xi_i),(\xi_j^*,0) \right]-J \left[(\xi_i^*,0),(\phi \xi_j^*, \xi_j) \right]= \left(-[\xi_i^*,\xi_j^*],[\xi_i,\xi_j] \right)
\end{aligned}
\end{align}

We see already that the necessary condition for the integrability is that $G$ must be abelian -- and so in fact $G=\mathbb{R}^p$. We proceed keeping that in mind.

\begin{align}
\begin{aligned}
[J,J] \left[(\xi_i^*,0),(0,\xi_j) \right] &= \left[(\phi \xi_i^*,\xi_i),(-\xi_j^*,0) \right]-J \left(0,[\xi_i,\xi_j] \right)-J \left(-[\xi_i^*,\xi_j^*],0 \right)\\
&=-\left(-[\xi_i,\xi_j]^*,0 \right)+ \left(\phi [\xi_i^*,\xi_j^*], \eta [\xi_i^*,\xi_j^*] \right)=0
\end{aligned}
\end{align}

\begin{align}
\begin{aligned}
[J,J] \left[(0,\xi_i),(0,\xi_j) \right]= - \left(0,[\xi_i,\xi_j] \right)+ \left([\xi_i^*,\xi_j^*],0 \right)=0
\end{aligned}
\end{align}

\begin{align}
\begin{aligned}
[J,J] \left[(X,0),(Y,0) \right]&=-\left([X,Y],0 \right)+ \left[(\phi X,0),(\phi Y, 0) \right]-J \left([\phi X,Y],0 \right)- J \left([X,\phi Y],0 \right) \\
&=\left(-[X,Y],0 \right)+ \left([\phi X, \phi Y],0 \right)- \left(\phi [\phi X,Y], \eta [\phi X,Y] \right)- \left(\phi [X,\phi Y],\eta [X,\phi Y] \right) \\
&= \left([\phi,\phi][X,Y]- \eta^* [X,Y] ,- \eta \left([\phi X,Y]+[X,\phi Y] \right) \right)
\end{aligned}
\end{align}

\begin{align}
\begin{aligned}
[J,J] \left[(X,0),(0,\xi_i) \right]&= \left[(\phi X,0),(- \xi_i^*,0) \right]-J \left[(\phi X,0),(0,\xi_i) \right]-J \left[(X,0),(-\xi_i^*,0) \right] \\
&=\left(-[\phi X, \xi_i^*],0 \right)+J \left([X,\xi_i^*],0 \right)= \left(-[\phi X, \xi_i^*],0 \right)+ \left(\phi [X,\xi_i^*], \eta [X,\xi_i^*] \right) \\
&= \left(\phi [X,\xi_i^*]-[\phi X, \xi_i^*], \eta [X,\xi_i^*] \right)
\end{aligned}
\end{align}

\begin{align}
\begin{aligned}
[J,J] \left[(X,0),(\xi_i^*,0) \right] &=- \left([X,\xi_i^*],0 \right)+ \left[(\phi X,0),(\phi \xi_i^*,\xi_i) \right]-J \left[(\phi X,0),(\xi_i^*,0) \right]-J \left[(X,0),(\phi \xi_i^*,\xi_i) \right] \\
&=\left(-[X,\xi_i^*]+[\phi X,\phi \xi_i^*],0 \right)- \left(\phi [\phi X,\xi_i^*], \eta [\phi X,\xi_i^*] \right)- \left(\phi [X,\phi \xi_i^*], \eta [X,\phi \xi_i^*] \right) \\
&=\left([\phi,\phi](X,\xi_i^*)- \eta^* [X,\xi_i^*] ,- \eta \left([\phi X,\xi_i^*]+[X,\phi \xi_i^*] \right) \right)
\end{aligned}
\end{align}

It turns out (as it is in the case of almost contact structures, \cite{bl}) that there is much interdependence between these cases. To be precise, the following theorem gives the necessary and sufficient conditions for an f.pk-$\g$-structure to be normal.

\begin{thm} \label{normal}
An f.pk-$\g$-structure $(M, \phi, \xi_i^*)$ is normal if and only if 

$$
\begin{array}{lr}
[\phi,\phi]+d \eta^*=0 & (\ast)
\end{array}
$$

where $d \eta^*$ is an antisymmetric tensor of type $(1,2)$ given by

$$
d \eta^* (Z,T) = d \eta_i (Z,T) \xi_i^*= \left( Z \eta_i T - T \eta_i Z - \eta_i [Z,T] \right) \xi_i^*
$$

This condition implies that $\g$ is abelian.
\end{thm}

\begin{proof}
We will show that the condition $(\ast)$ is necessary. There are only three cases to consider.

We start with 

$$
[\phi,\phi](\xi_i^*,\xi_j^*)+d \eta^* (\xi_i^*,\xi_j^*)=\left(-[\xi_i^*,\xi_j^*],[\xi_i,\xi_j] \right)=0
$$

because of $(I)$. If $X,Y$ are vector fields with values in $D$, from $(IV)$ we have that

$$
0 = [\phi,\phi](X,Y) - \eta^* [X,Y] =[\phi,\phi](X,Y) + d 
\eta^* (X,Y)
$$

since $\eta_i X = \eta_i Y = 0$. From $(VI)$ we have

$$
0 = [\phi,\phi](X,\xi_i^*) - \eta^* [X, \xi_i^*] = [\phi,\phi](X,\xi_i^*) + d \eta^* (X,\xi_i^*)
$$

because $\eta_j X=0$ and $\eta_j \xi_i^* =const$ and we are done.

We now show that the condition $(\ast)$ is sufficient. We have already seen at the beginning of this proof that under $(\ast)$ the group $G$ is abelian and consequently $(I),(II)$ and $(III)$ vanish. Observe that $\eta^*(d \eta^*) = d \eta^*$ and so composing each side of $(\ast)$ with $\eta^*$ gives

$$
\eta^* \left([\phi v, \phi w] \right) + d \eta^* (v,w) = 0
$$

Taking $v=X$ and $w=\xi_i^*$ we get

$$
0=d \eta^* (X,\xi_i^*)=- \eta^* [X, \xi_i^*]
$$

and so the second coordinate in $(V)$ is 0. We also see that the following vanishes

$$
[\phi,\phi](X,\xi^*_i)= \phi^2 [X, \xi_i^*]+ [\phi X, \phi \xi_i^*]- \phi [X,\phi \xi^*_i]- \phi [\phi X, \xi_i^*]=\phi \left( \phi [X, \xi_i^*]-[\phi X, \xi_i^*] \right)
$$

and composing with $\phi$ leads us to

$$
0= -\phi[X,\xi_i^*]+[\phi X,\xi_i^*]+ \eta^* \phi[X,\xi_i^*]- \eta^* [\phi X,\xi_i^*]=-\phi[X,\xi_i^*]+[\phi X,\xi_i^*]
$$ 

since $\eta^* \phi = 0$ and as we have seen $\eta^* [\phi X, \xi_i^*]=0$. We have shown that also the first coordinate in $(V)$ is 0. Again starting from

$$
[\phi , \phi](\phi v, w)+ d \eta^* (\phi v, w)=0
$$ 

and taking $\eta^*$ of both sides we get

$$
\eta^* [\phi^2 v, \phi w] = - \eta^* [v, \phi w] + \eta^* [\eta^* v, \phi w] = \eta^* [\phi v, w]
$$

which shows the vanishing of the second coordinate in $(IV)$ and~$(VI)$. The first coordinate in $(IV)$ and $(VI)$ already vanished by our assumptions, so we are done.
   
\end{proof}

\section{On 3-dimensional f.pk-$\mathfrak{g}$-structures}

The obvious way of defining an almost complex structure on the manifold $M \times G$ together with the integrability condition forced the Lie algebra $\g$ to be abelian. In the following section we will therefore look for alternative options for defining an almost complex structure and describe the implications of the integrability condition in that case.

Throughout this section assume that $\g$ is a 3-dimensional Lie algebra of a simply connected Lie group $G$ and spanned by $\{\xi_1,\xi_2, \xi_3\}$. Recall that these are classified into 9 types by Bianchi. Let $(M,\phi, \xi_1^*, \xi_2^*, \xi_3^*)$ be an f.pk-$\g$-structure. Let $\eta_1$, $\eta_2$, $\eta_3$ be the 1-forms as in the previous section and $\lambda_1$, $\lambda_2$, $\lambda_3$ denote linear forms dual to left invariant fields $\xi_1$, $\xi_2$, and $\xi_3$, respectively. We try to construct an almost complex structure on $M \times G$ that stands a chance of being integrable. 

In the presence of an f.pk-$\g$-structure, the suggestive approach to take is to find an integrable (ie. the natural Nijenhuis bracket vanishes) complex structure $\mathcal{J}$ on $\g^* \oplus \g$. We know however we can not take the simplest route $\xi_1^* \mapsto \xi_1 $, $\xi^*_2 \mapsto \xi_2$, $\xi_3^* \mapsto \xi_3$, so instead consider the one given by

\begin{align*}
\mathcal{J} \xi_1&=\xi_2\\
\mathcal{J} \xi_1^*&=\xi_2^* \\
\mathcal{J} \xi_3^*&=\xi_3.
\end{align*}

As shown in \cite{MY}, only seven of the nine types of 3-dimensional Lie algebras $\g$ give a product algebra $\g\oplus\g$ that admit an integrable complex structure at all -- but each then admits a structure of this form. These algebras include $\eseldwa$, $\otrzy$ and the Heisenberg algebra $\hatrzy$ and for convenience they will be properly listed in Remark \ref{lista}. The following theorem gives a necessary and sufficient condition on the parallelism $\xi_1^*, \xi_2^*, \xi_3^*$ for $\mathcal{J}$ to be integrable.

\begin{thm} \label{dobre}
The almost complex structure $\mathcal{J}$ on $\g^* \oplus \g$ is integrable if and only if the adjoint endomorphism $[\xi^*_3,\cdot ]$ has either a real eigenvalue $\gamma$ to which $\xi^*_1, \xi^*_2$ are linearly independent eigenvectors, or a complex one, $\alpha+\beta i$ with conjugated eigenvectors $v, w,$ -- and then $\xi_1^*=av+bw$ for some $a,b \in \mathbb{R}$ and $\xi_2^* = bv-aw$ or $\xi_2^* = -bv+aw$.
\end{thm}

\begin{proof} Suppose $\mathcal{J}$ is an integrable complex structure on $\g^* \oplus \g$. Nijenhuis tensor being zero for $\xi^*_3$ and any other $\xi^*$ in $\g^*$ gives

$$
-[\xi^*_3,\xi^*]+[\mathcal{J} \xi^*_3, \mathcal{J} \xi^*] - \mathcal{J}[\mathcal{J} \xi^*_3, \xi^*] - \mathcal{J}[\xi^*_3, \mathcal{J} \xi^*] = -[\xi^*_3,\xi^*] - \mathcal{J}[\xi^*_3, \mathcal{J} \xi^*]=0
$$

Taking $\mathcal{J}$ of both sides

$$
\mathcal{J} [\xi^*_3,\xi^*] = [\xi^*_3, \mathcal{J} \xi^*]
$$

We will exploit that in a moment, but we must exclude several degenerate cases first. Suppose that $[\xi^*_3,\cdot ]$ has only zero eigenvalue but $[\xi^*_3,\cdot ] \not \equiv 0$. In a Jordan basis $\{\xi_3^*,v,w\}$ (note we can always include $\xi^*_3$ into basis) the adjoint can then take three forms. If

$$
[\xi^*_3,\cdot ]=
\left[
\begin{array}{ccc}
0 & 0 & 0 \\
0 & 0 & 1 \\
0 & 0 & 0
\end{array}
\right]
$$

then $\mathcal{J} v = \mathcal{J}[\xi_3^*,w]=[\xi_3^*\mathcal{J}w]$ which gives $\mathcal{J} v \in span \{v \}$, a contradiction ($\mathcal{J}$ does not have real eigenvalues).
If a Jordan form is

$$
[\xi^*_3,\cdot ]=
\left[
\begin{array}{ccc}
0 & 1 & 0 \\
0 & 0 & 0 \\
0 & 0 & 0
\end{array}
\right] \text{ or }
[\xi^*_3,\cdot ]=
\left[
\begin{array}{ccc}
0 & 1 & 0 \\
0 & 0 & 1 \\
0 & 0 & 0
\end{array}
\right]
$$

we argue that $\mathcal{J} \xi_3^* = \mathcal{J} [\xi_3^*, v] = [\xi_3^*, \mathcal{J}v] \in \g^*$, a contradiction. Thus if $\mathcal{J}$ is integrable and $0$ is an eigenvalue of $[\xi^*_3,\cdot ]$ then $\xi_1^*, \xi_2^*$ are in its kernel.

Now assume that $\gamma$ is a non-zero real eigenvalue with eigenvector $v$. Then

$$
\gamma \mathcal{J} v = \mathcal{J} [\xi_3^*, v] = [\xi_3^*, \mathcal{J}v] \in \g
$$

and so $\mathcal{J}$ is another eigenvector for $\gamma$, necessary linearly independent of $v$, again because $\mathcal {J}$ does not posses real eigenvalues. We also see that $span\{v ,\mathcal{J} v \}$ is an $\mathcal{J}$-invariant subspace of $\g^*$ and consequently $span\{v ,\mathcal{J} v \} = span \{\xi_1^*, \xi ^*_2 \}$ which proves our assertion about $\xi_1^*$ and $\xi ^*_2$.

Finally assume that $[\xi_3^*, \cdot]$ has a complex eigenvalue $\alpha+\beta i$ with conjugated eigenvectors $v$ and $w$ -- then

$$
\alpha \mathcal {J} v +\beta  \mathcal{J} w =\mathcal{J}[\xi_3^*,v] = [\xi_3^*,\mathcal{J} v]
$$

and

$$
-\beta \mathcal {J} v +\alpha \mathcal{J} w =\mathcal{J}[\xi_3^*,w] = [\xi_3^*,\mathcal{J} w]
$$

We easily get $\mathcal{J} v, \mathcal{J} w \in \g^*$ and that they too are conjugated eigenvectors for $\alpha+\beta i$. Consequently, if $\mathcal{J} v = x v + y w$ then $\mathcal{J} w = -yv+xw$, but since $\mathcal{J}^2=-Id$, $\mathcal{J}w = \frac{-1-x^2}{y}v-xw $ and that gives $x=0$ and $y=\pm 1$. Since $span \{v, w\}$ is $\mathcal{J}$-invariant subspace of $\g$, again $span \{v,w\} = span \{\xi_1^*, \xi_2^* \}$. If $\xi_1^*$ is $av+bw$ then $\mathcal{J} \xi_1^* = \xi_2^* = -bv+aw$ or $\mathcal{J} \xi_1^* = \xi_2^* = bv-aw$ as we wanted.

The opposite implication is easily checked.
\end{proof}

\begin{rmk} \label{redu}
As a consequence of this proof, the almost complex structure $\mathcal{J}$ is integrable if and only if

$$
[\mathcal{J},\mathcal{J}][\xi^*_3,\xi^*_1]=[\mathcal{J},\mathcal{J}][\xi^*_3,\xi^*_2]=0
$$

The symmetric conditions on $\g$ are then automatically satisfied because inverse of $^*$ preserves the eigenvectors of adjoint endomorphisms.
\end{rmk}

\begin{rmk}{\cite{MY}}\label{lista}
A 3-dimensional Lie algebra $\g$ admits a basis $\xi_1, \xi_2, \xi_3$ such that the adjoint $[\xi_3,\cdot]$ has a real eigenvalue, (to which $\xi_1$ and $\xi_2$) are two linearly independent eigenvectors or a complex eigenvalue (to which they are conjugate complex eigenvectors) if and only if $\g$ is of one of seven following types, given by their multiplication tables

\begin{enumerate}
\item $[\xi_1,\xi_3]= 0$, $[\xi_2,\xi_3]= 0$, $[\xi_1,\xi_2]= 0$ 
\item $[\xi_1,\xi_3]= 0$, $[\xi_2,\xi_3]= 0$, $[\xi_1,\xi_2]= \xi_1$
\item $[\xi_1,\xi_3]= 0$, $[\xi_2,\xi_3]= 0$, $[\xi_1,\xi_2]= \xi_3$ 
\item $[\xi_1,\xi_3]= \xi_1$, $[\xi_2,\xi_3]= \xi_2$, $[\xi_1,\xi_2]= 0$ 
\item $[\xi_1,\xi_3]= \theta \xi_1-\xi_2$, $[\xi_2,\xi_3]= \xi_1+\theta \xi_2$, $[\xi_1,\xi_2]= 0$  for $\theta \neq 0$
\item $[\xi_1,\xi_3]=\xi_2$, $[\xi_2,\xi_3]= -\xi_1$, $[\xi_1,\xi_2]= \xi_3$
\item $[\xi_1,\xi_3]=-\xi_2 $, $[\xi_2,\xi_3]=\xi_1 $, $[\xi_1,\xi_2]= \xi_3$
\end{enumerate}

The given bases satisfy the condition.

\end{rmk}

We will now implement the structure $\mathcal{J}$ into our structure tensors of the f.pk-$\g$-manifold. Denote by $\hat{\phi}$ a $(1,1)$-tensor on $M$ given by

\begin{align*}
\left.\hat{\phi}\right|_{D}&=\phi\\
\hat{\phi}\xi_1^*&=\xi_2^*\\
\hat{\phi}\xi_2^*&=-\xi_1^*\\
\hat{\phi}\xi_3^*&=0
\end{align*}

and by $\psi$ a $(1,1)$ tensor on $G$ given by

\begin{align*}
\psi\xi_1&=\xi_2\\
\psi\xi_2&=-\xi_1\\
\psi\xi_3&=0
\end{align*}

Then we have $\hat{\phi}^2=-I+\eta_3 \otimes  \xi_3^*$ and $\hat{\phi}$ is a new f.pk-structure on $M$ of codimension one. We define an almost complex structure

$$
\hat{J}(X,Y)= \left(\hat{\phi} (X) - \lambda_3(Y) \xi_3^*, \psi (Y) + \eta_3(X) \xi_3 \right)
$$

with the usual convention that $(X,Y) \in \chi (M \times G)$ denotes, respectively, the $TM$- and $TG$-part of a vector field. Observe that $\left.\hat{J}\right|_{\g^*\oplus\g}\equiv \mathcal{J}$ above.

We are now ready to give the main definition of this section.

\begin{de}
An f.pk-$\g$-structure $(M,\phi, \xi_1^*, \xi_2^*, \xi_3^*)$ is \emph{mixed normal} if and only if $\hat{J}$ is integrable.
\end{de}

The word "mixed" alludes to the structure $\mathcal{J}$ mixing the two copies of $\g$ -- a single direction is interchanged between them, while each copy contains a non-trivial complex subspace. Our main theorem is the characterisation of this condition.

\begin{thm}\label{Heisenberg}
The almost complex structure $\hat{J}$ is integrable if and only if 

$$
\begin{array}{lr}
[\hat{\phi},\hat{\phi}]+d\eta_3\otimes \xi_3^*=0 & (\ast\ast)
\end{array}
$$

\end{thm}

\begin{proof}

As in Theorem \ref{normal} above, it is sufficient to examine the Nijenhuis tensor in five cases, this time coming from the splitting 

$$
T(M \times G) = D \oplus span \{\xi_1^*, \xi_2^* \} \oplus span \{\xi_3^* \} \oplus span \{ \xi_1, \xi_2, \xi_3 \} = \hat D \oplus span \{\xi_3^* \} \oplus span \{ \xi_1, \xi_2, \xi_3 \}
$$

suitable for a new f.pk-structure of codimension one. They are

\setcounter{equation}{0}

\begin{align}
\begin{aligned}
[\hat{J},\hat{J}] & \left[(\xi_3^*,0),(0,\xi_i) \right] \\
\end{aligned}
\end{align}

\begin{align}
\begin{aligned}
[\hat{J},\hat{J}] & \left[(0,\xi_i),(0,\xi_j) \right] \\
\end{aligned}
\end{align}

\begin{align}
\begin{aligned}
[\hat{J},\hat{J}] & \left[(X,0),(Y,0) \right]=-\left([X,Y],0 \right)+ \left[(\hat{\phi} X,0),(\hat{\phi} Y, 0) \right]- \hat{J} \left([\hat{\phi} X,Y],0 \right)- \hat{J} \left([X,\hat{\phi} Y],0 \right)\\
&=\left(-[X,Y],0 \right)+ \left([\hat{\phi} X, \hat{\phi} Y],0 \right)- \left(\hat{\phi} [\hat{\phi} X,Y], \eta_3 [\hat{\phi} X,Y] \xi_3 \right)-\left(\hat{\phi} [X,\hat{\phi} Y],\eta_3 [X,\hat{\phi} Y] \xi_3 \right)\\
&=\left([\hat{\phi},\hat{\phi}][X,Y]-  \eta_3 [X,Y] \xi_3^*,- \eta_3 \left([\hat{\phi} X,Y]+[X,\hat{\phi} Y] \right) \xi_3 \right)
\end{aligned}
\end{align}

\begin{align}
\begin{aligned}
[\hat{J},\hat{J}] & \left[(X,0),(0,\xi_i) \right]= \left[(\hat{\phi} X,0),(- \lambda_3 \xi_i \xi_3^*,\psi \xi_i ) \right]-\hat{J} \left[(\hat{\phi} X,0),(0,\xi_i) \right]-\hat{J} \left[(X,0),(-\lambda_3 \xi_i  \xi_3^*, \psi \xi_i ) \right] \\
&=\left(-[\hat{\phi} X, \lambda_3 \xi_i \xi_3^*],0 \right)+ \hat{J} \left([X, \lambda_3 \xi_i \xi_3^*],0 \right) \\
&= \left(-[\hat{\phi} X, \lambda_3 \xi_i \xi_3^*],0 \right)+ \left(\hat{\phi} [X, \lambda _3 \xi_i \xi_3^*], \eta_3 [X,\lambda_3 \xi_i \xi_3^*] \xi_3 \right) \\
&= \left(\hat{\phi} [X,\lambda_3 \xi_i \xi_3^*]-[\hat{\phi} X, \lambda_3 \xi_i \xi_3^*], \eta_3 [X,\lambda_3 \xi_i \xi_3^*] \xi_3 \right)
\end{aligned}
\end{align}

\begin{align}
\begin{aligned}
[\hat{J},\hat{J}] & \left[(X,0),(\xi_3^*,0) \right]\\
&=- \left([X,\xi_3^*],0 \right)+ \left[(\hat{\phi} X,0),(0,\xi_3) \right]-\hat{J} \left[(\hat{\phi} X,0),(\xi_3^*,0) \right]- \hat{J} \left[(X,0),(0, \xi_3) \right] \\
&=\left(-[X,\xi_i^*],0 \right)- \left(\hat{\phi} [\hat{\phi} X, \xi_3^* ], \eta_3 [\phi X,\xi_3^*] \xi_3 \right)- \left(\hat{\phi} [X,\hat{\phi} \xi_3^*], \eta_3 [X,\hat{\phi} \xi_3^*] \xi_3 \right) \\
&=\left([\hat{\phi},\hat{\phi}](X,\xi_3^*)-  \eta_3 [X,\xi_3^*]  \xi_3^*,- \eta_3 \left([\phi X,\xi_3^*]+[X,\phi \xi_3^*] \right) \xi_3 \right) 
\end{aligned}
\end{align}

where $X,Y$ have values in $\hat{D}$. We see that the condition $(\ast\ast)$ is necessary for the expressions $(I)$-$(V)$ to vanish since we can recover the tensor $[\hat{\phi},\hat{\phi}]+d\eta_3\otimes \xi_3^*$ from the first coordinates of $(III)$ and $(V)$.

\begin{rmk} \label{cond}
For future reference, we stress that $(\ast\ast)$ is trivially satisfied on $\g^*$ if $\mathcal{J}$ is integrable.
\end{rmk}

Proving the condition sufficient follows the same steps as for Theorem \ref{normal}, except the following cases. In $(IV)$ $[\hat{J},\hat{J}]\left((X,0),(0,\xi_1)\right)$ and $[\hat{J},\hat{J}]\left((X,0),(0,\xi_2)\right)$ require separate treatment -- these vanish because $\lambda_3(\xi_1)=\lambda_3(\xi_2)=0$. Finally, we know that $[\hat{J},\hat{J}]  \left[(\xi_3^*,0),(\xi_1^*,0) \right]=[\hat{J},\hat{J}]  \left[(\xi_3^*,0),(\xi_2^*,0) \right]=0$ and that, by Remark \ref{redu} gives integrability of $\mathcal{J}$ and consequently vanishing of $(I)$ and $(II)$.
\end{proof}

If the parallelism is not given by a group acting on the manifold, we choose the simply connected $G$ for definiteness only. However, if some group does act -- as in the following constructions -- we will use this group instead. This does not change the theorem or proof in any way.

We stress that although we have formulated the mixed normality condition using $M\times G$ due to its roots in Sasakian geometry, since the tensor calculus is local, the same proof works in the two following locally trivial cases. They will serve as important examples below and seem to be well suited for physical applications.

Let $G^3\hookrightarrow M^{2n+3}\longrightarrow W$ be a principal bundle given by a cocycle $\{\Gamma_{\alpha\beta}\}$. Let $(M,\phi,\xi_1^*,\xi_2^*,\xi_3^*)$ be an f.pk-$\g$-structure, with $^*$ coming from the action of $G$ on $M$. We can construct a new principal $G\times G$-bundle over $W$ using the cocycle $\{(\Gamma_{\alpha\beta},\Gamma_{\alpha\beta})\}$. We call this bundle $M\oplus G$ in contrast with $M\times G$ before. The $G\times G$ action gives six global vector fields, $\{\xi_1^*,\xi_2^*,\xi_3^*,\xi_1,\xi_2,\xi_3\}$. By local triviality -- and since each distribution is preserved by the new cocycle -- we can still write the decomposition $T\left(M\oplus G\right)=D \oplus span\{\xi_i^*\} \oplus span\{ \xi_i\}$.  Define $\hat{\phi}$, $\psi$, $\hat{J}$ on $M\oplus G$ in an analogous fashion as before. The theorem follows from the proof of Theorem \ref{normal}.

\begin{thm}\label{bundle}
The almost complex manifold $\left(M\oplus G, \hat{J} \right)$ is complex if and only if the condition $(\ast\ast)$ holds.
\end{thm}

Now suppose that a compact group $G^3$ acts locally freely on $M^{2n+3}$ and suppose that $(M,\phi,\xi_1^*,\xi_2^*,\xi_3^*)$ is an f.pk-$\g$-structure with $^*$ coming from the action. Then $M$ is again a locally trivial bundle, but this time it is an orbifold bundle over some orbifold $W$. Given its defining cocycle $\{ \Gamma_{\alpha\beta}\}$, we again construct a new orbifold bundle over $W$, with fiber $G\times G$ using the cocycle $\{(\Gamma_{\alpha\beta},\Gamma_{\alpha\beta})\}$. We call the total space $M\rtimes G$, and stress that it is again a manifold. Its tangent bundle splits as before, into $D \oplus span \{ \xi_i^* \} \oplus span \{ \xi_i \}$.  Adjusting the definitions once more, we get tensors $\hat{\phi}$, $\psi$, $\hat{J}$ on $M\rtimes G$, and the theorem characterising the integrability of $\hat{J}$.

\begin{thm}\label{bundlee}
The almost complex manifold $\left(M\rtimes G, \hat{J} \right)$ is complex if and only if the condition $(\ast\ast)$ holds.
\end{thm}

From the three theorems of this section, Theorem \ref{normal} and definitions we can formulate the following theorem which describes connections between normal and mixed normal structures.

\begin{thm} \label{normalnormal}
For a manifold $(M,\phi,\xi_1^*,\xi_2^*,\xi_3^*)$ with an f.pk-$\g$-structure, the following conditions are equivalent:

\begin{enumerate}
\item $[\hat{\phi},\hat{\phi}]+ d\eta_3 \otimes \xi_3^*=0$;
\item $\left(M,\phi,\xi_i^*\right)$ is a mixed normal f.pk-$\g$-structure;
\item $\left (M,\hat{\phi},\xi_3^*\right)$ is a normal f.pk-structure;
\item the almost complex structure $J$ on $M\times\mathbb{R}$ is integrable;
\item the almost complex structure $\hat{J}$ on $M \times G$ is integrable;
\item if $M$ is a principal $G$-bundle with parallelism given by invariant fields, the almost complex structure $\hat{J}$ on $M \oplus G$ is integrable;
\item if a compact Lie group $G$ acts locally freely on $M$ and parallelism is given by invariant fields, the almost complex structure $\hat{J}$ on $M \rtimes G$ is integrable.
\end{enumerate}

\end{thm}

\section{Examples}

We have made some assumptions about the structure tensor $\phi$ on a $\g$-manifold. We would now like to give some examples to show that they happen to be fulfilled. 

\begin{ex} The lowest possible dimension is three. Let $M=G^3$ -- a group (not necasarry simply connected) with a Lie algebra $\g$ as in Remark \ref{lista}. Take a trivial f.pk-$\g$-structure $\phi\equiv 0$ and any parallelism satisfying the conditions in Theorem \ref{dobre}. Then $\hat{J}$ is the integrable left invariant complex structure on $G \times G$. We include the list of possible groups $G$ here, for reference

\begin{itemize}
\item the abelian group $\mathbb{R}^3$;
\item the group of invertible upper-triangular $2\times 2$ matrices with $\g=\mathfrak{t}(2)$, all upper-triangular $2\times 2$ matrices;
\item the Heisenberg group $H_3$ with $\g=\hatrzy$;
\item the 2-dimensional Poincar\'e group $P(1,1)$ with $\g=\mathfrak{p}(1,1)$ (sometimes denoted $\mathfrak{iso}(1,1)$);
\item affine isometries of $\mathbb{R}^2$ with $\g=\mathfrak{e}(2)$; 
\item the special linear group $SL(2,\mathbb{R})$ with $\g=\eseldwa$;
\item the orhogonal group $O(3)$ with $\g=\otrzy$;
\end{itemize}

We point out that 6 is the smallest dimension where the classification of Lie groups admitting left invariant complex structures is not yet known, cf. \cite{smol}. This example together with \cite{MY} gives a full picture for 6-dimensional product groups $G \times G$.  
\end{ex}

\begin{ex} \label{ex}
First non-trivial example is obtained in dimension $5$. Consider a Lie algebra 

$$
\mathfrak{g}=\left\{
\left[\begin{array}{ccccc} 
0 & 0 & p^* & s & r^* \\
0 & 0 & 0 & 0 & 0 \\
0 & 0 & 0 & 0 & q^* \\
0 & 0 & 0 & 0 & t \\
0 & 0 & 0 & 0 & 0
\end{array}\right] \quad | \quad p^*,q^*,r^*,s,t \in \mathbb{R} \right\}
$$

We will abuse notation slightly to write $p^*$ for the matrix with single 1 at the $p^*$-entry above and so on. The bracket structure is easily checked to be $[s,t]=r^*$, $[p^*,q^*]=r^*$ and 0 otherwise. Note that $\hatrzy= span \{p^*,q^*,r^* \}$ is embedded in $\g$. By the Lie algebra - Lie group correspondence we find a group $G$ with subgroup $H_3$. Thus $H_3$ acts freely on $G$ and we define an f.pk-$\hatrzy$-structure $(G, \phi, p^*,q^*,r^*)$ by

\begin{align*}
\phi s = t \\
\phi t = -s \\
\left.\phi\right|_{span\{p^*,q^*,r^* \}} &\equiv 0
\end{align*}

We check the condition $(\ast\ast)$. Since $D$ in this case is 2-dimensional parallelisable distribution we only need to compute

\begin{align*} 
[\hat{\phi},\hat{\phi}](s,t)+ d\eta_3[s,t] r^*=[t,-s]-r^*=0
\end{align*}

and

\begin{align*}
[\hat{\phi},\hat{\phi}](v,w)+ d\eta_3[v,w] r^*
\end{align*}

for any $v\in \{s,t \}$ and $w \in \{ p^*, q^*, r^*\}$ -- but then $[v,w]=0$ and the expressions vanish. This example, though not complicated, has the property that the distribution $D=im{}\phi$ is non-integrable, and so the normality does supply an additional geometric information.
\end{ex}

\begin{ex} This one is a non-example, in fact. In a similar vein, take the Lie algebra 

$$
\mathfrak{g}=\left\{
\left[\begin{array}{ccccc} 
s & 0 & p^* & 0 & r^* \\
0 & 0 & 0 & 0 & 0 \\
0 & 0 & 0 & 0 & q^* \\
0 & 0 & 0 & 0 & 0 \\
0 & 0 & 0 & 0 & t
\end{array}\right] \quad | \quad p^*,q^*,r^*,s,t \in \mathbb{R} \right\}.
$$

The bracket structure is $[s,r^*]=[r^*,t]=r^*$, $[p^*,q^*]=r^*$ and 0 otherwise. For this $\g$, we again find a group $G$ with a subgroup $H_3$ and define an f.pk-$\hatrzy$-structure $(G, \phi, p^*,q^*,r^*)$ as in Example \ref{ex}. The condition $(\ast\ast)$ is not satisfied because 

\begin{align*}
[\hat{\phi},\hat{\phi}](s,r^*)+ d\eta_3[s,r^*] r^*=-\eta_3 [s,r^*] r^*=-r^* \not = 0.
\end{align*}

The geometric reason for $(\ast\ast)$ to fail is that for a mixed normal structure, $r^*$ must preserve the distribution $\hat{D}$ -- which is not the case here.
\end{ex}

\begin{ex}
Consider an orientable, genus $g$ surface $\Sigma$. Any volume form $\omega$ gives rise to a $Sl(2,\mathbb{R})$ structure, and a $Sl(2,\mathbb{R})$-bundle of frames $F\Sigma$. This volume form also has a compatible almost complex structure $\mathbb{J}$, necessarily integrable, since the dimension is 2. As we already reduced the structure group of the bundle to $U(1)$, a maximal torus. We may assume without loss of generality that this $U(1)$ is the subgroup of rotations inside determinant 1 matrices, since all maximal tori are conjugate.

The tangent bundle $T\Sigma$ admits an affine connection with holonomy in $\mathfrak{u}(1)$, and this defines a connection $\nabla$ in the principal bundle $F\Sigma$. We will use that for any two vertical fields $X$ and $Y$ the horizontal part of $[X,Y]$ lies in $\mathfrak{u}(1)$.


We proceed to define an f.pk-structure. We will identify $\mathfrak{u}(1)$ in $\eseldwa$ with the algebra generated by $\xi_3^*=\left[\begin{array}{cc} 0 & 1 \\ -1 & 0 \end{array}\right]$. This vector is an appropriate choice for $\xi_3^*$ in Theorem \ref{dobre} -- we pick a Jordan basis of $[\xi_3^*,\cdot]$ as our parallelism $\{\xi_1^*,\xi_2^*,\xi_3^*\}$. We can now define $\phi$ on $F\Sigma$ by $\phi(X)=\left(\mathbb{J} d\pi X\right)^{\nabla}$, where $\pi$ is the projection in the bundle and $v^{\nabla}$ is the unique horizontal lift defined by the connection -- and $(F\Sigma,\phi,\xi_1^*,\xi_2^*,\xi_3^*)$ is an f.pk-$\eseldwa$-structure.

We will now prove that this structure is mixed normal. We define $\hat{\phi}$ and $\hat{J}$ on $F\Sigma\oplus Sl(2,\mathbb{R})$ as before, and we check the condition $(\ast\ast)$, first on the horizontal fields -- which we can take to satisfy $\phi X=Y$.

\begin{align*}
[\hat{\phi},\hat{\phi}](X,Y)+&d\eta_3\otimes \xi_3^*(X,Y)\\
&=\hat{\phi}^2[X,Y]+[\hat{\phi} X,\hat{\phi} Y]-\hat{\phi}[\hat{\phi} X,Y]-\hat{\phi}[X, \hat{\phi} Y]+\left(X\eta_3 Y-Y\eta_3X-\eta_3[X,Y]\right)\xi_3^*\\
&=\left(\mathbb{J}^2\pi[X,Y]\right)^{\nabla}+[Y,-X]-\left(\mathbb{J}\pi[\phi X,Y]\right)^{\nabla}-\left(\mathbb{J}\pi[ X,\phi Y]\right)^{\nabla}-\eta_3[X,Y]\xi_3^*\\
&=\left(\mathbb{J}^2\pi[X,Y]\right)^{\nabla}+[\pi Y,\pi-X]^{\nabla}+x\xi_3^*-\left(\mathbb{J}[\mathbb{J}\pi X,\pi Y]\right)^{\nabla}-\left(\mathbb{J}\pi[\pi X,\mathbb{J}\pi Y]\right)^{\nabla}-x\xi_3^*\\
&=\left([\mathbb{J},\mathbb{J}](\pi X,\pi Y)\right)^{\nabla} =0
\end{align*}

because $\mathbb{J}$ was integrable. As we mentioned before $(\ast\ast)$, is trivially satisfied on $\g^*$, so we are only left with the mixed horizontal/vertical pairs. But for each vertical $\xi^*$ and horizontal $X$, we have $[X,\xi^*]=0$ and so

\begin{align*}
[\hat{\phi},\hat{\phi}](X,\xi^*)+&d\eta_3\otimes \xi_3^*(X,\xi^*)\\
&=\hat{\phi}^2[X,\xi^*]+[\hat{\phi} X,\hat{\phi} \xi^*]-\hat{\phi}[\hat{\phi} X,\xi^*]-\hat{\phi}[X, \hat{\phi} \xi^*]+\left(X\eta_3 \xi^*-\xi^*\eta_3X-\eta_3[X,\xi^*]\right)\xi_3^*\\
&=-\eta_3[X,\xi_3^*]\xi_3^*=0
\end{align*}

This example also features a non-integrable distribution $D$ (only genus-one torus admits a flat connection). We point that $Sl(2,\mathbb{R})$ is isomorphic to $O(2,1)$ so we feel this example may be relevant to the study of (2+1)-dimensional gravity. Although mixed normality of $\left(F\Sigma,\phi,\xi_1^*,\xi_2^*,\xi^*_3\right)$ is encoded by either $F\Sigma\times G$ or $F\Sigma\oplus G$, the latter seems to be closer realted to overal geometry of $F\Sigma$ (or of $\Sigma$ itself).

\end{ex}


\end{document}